\documentclass[a4paper]{amsart}
\usepackage{a4wide}

\newtheorem{theorem}{Theorem}[section]

\theoremstyle{definition}

\newtheorem{mthm}[theorem]{Main Theorem}

\theoremstyle{remark}

\numberwithin{equation}{section}

\usepackage{amsmath}
\usepackage{amsfonts}
\usepackage{amssymb}
\usepackage{amsthm}

\usepackage{cite}

\usepackage[mathscr]{eucal}

\newtheorem{lem}{Lemma}[section]

\numberwithin{equation}{section}

\newtheorem*{cor*}{Corollary}
\newtheorem*{thm*}{Theorem}

\theoremstyle{definition}
\newtheorem{defi}{Definition}[section]

\theoremstyle{remark}
\newtheorem{rem}[lem]{Remark}

\newcommand{\labl}[1]{\label{lemma:#1}}
\newcommand{\refl}[1]{Lemma~\ref{lemma:#1}}
\newcommand{\labmt}[1]{\label{mthm:#1}}
\newcommand{\refmt}[1]{Main Theorem~\ref{mthm:#1}}
\newcommand{\labeq}[1]{\label{eqn:#1}}
\newcommand{\refeq}[1]{(\ref{eqn:#1})}

\allowdisplaybreaks[3]

\newcommand{\N}{\mathbb{N}}

\newcommand{\lf}{\left\lfloor}
\newcommand{\rf}{\right\rfloor}
\newcommand{\lc}{\left\lceil}
\newcommand{\rc}{\right\rceil}
\renewcommand{\lvert}{\left\vert}
\renewcommand{\rvert}{\right\vert}

\renewcommand{\tilde}{\widetilde}

\newcommand{\floor}[1]{\left\lfloor #1 \right\rfloor} 
\newcommand{\ceil}[1]{\left\lceil #1 \right\rceil}

\newcommand{\adkmu}{\mathscr{D}_{\mu,k}}
\newcommand{\admu}{\mathscr{D}_{\mu}}

\newcommand{\womega}{\ell}

\newcommand{\nuib}{\nu_i(\mathbf{b})}
\newcommand{\nuipb}{\nu_{i+1}(\mathbf{b})}

\newcommand{\ei}{\varepsilon_i}

\newcommand{\eim}{\varepsilon_{i-1}}
\newcommand{\li}{\ell_i}
\newcommand{\lip}{\ell_{i+1}}
\newcommand{\limo}{\ell_{i-1}}
\newcommand{\lxi}{|\mathbf{w}_i|}
\newcommand{\lxip}{|\mathbf{w}_{i+1}|}
\newcommand{\lxim}{|\mathbf{w}_{i-1}|}

\newcommand{\Nmuk}[2]{\mathscr{N}_{ {#1}, {#2} } }

\newcommand{\Nm}{\mathscr{N}_{\mu }}

\begin{document}

\title{Construction of $\mu$-normal sequences}

\author[M. G. Madritsch]{Manfred G. Madritsch}
\address[M. G. Madritsch]{1. Universit\'e de Lorraine, Institut Elie Cartan de Lorraine, UMR 7502, Vandoeuvre-l\`es-Nancy, F-54506, France;\newline
2. CNRS, Institut Elie Cartan de Lorraine, UMR 7502, Vandoeuvre-l\`es-Nancy, F-54506, France}
\email{manfred.madritsch@univ-lorraine.fr}

\author[B. Mance]{Bill Mance} \address[B. Mance]{ Department of
  Mathematics, University of North Texas, General Academics Building
  435, 1155 Union Circle \#311430, Denton, TX 76203-5017, USA}
\email{mance@unt.edu}

%    Information for first author
%\author{Author One}
%    Address of record for the research reported here
%\address{Department of Mathematics, Louisiana State University, Baton
%Rouge, Louisiana 70803}
%    Current address
%\curraddr{Department of Mathematics and Statistics,
%Case Western Reserve University, Cleveland, Ohio 43403}
%\email{xyz@math.university.edu}
%    \thanks will become a 1st page footnote.
%\thanks{The first author was supported in part by NSF Grant \#000000.}

%    Information for second author
%\author{Author Two}
%\address{Mathematical Research Section, School of Mathematical Sciences,
%Australian National University, Canberra ACT 2601, Australia}
%\email{two@maths.univ.edu.au}
%\thanks{Support information for the second author.}

%    General info

\date{\today}

%\dedicatory{This paper is dedicated to our advisors.}


\begin{abstract}
  In the present paper we extend Champernowne's construction of normal
  numbers to provide sequences which are generic for a given invariant
  probability measure, which need not be the maximal one. We present a
  construction together with estimates and examples for normal numbers
  with respect to L\"uroth series expansion, continued fractions
  expansion or $\beta$-expansion.
\end{abstract}

\maketitle

\section{Introduction}
Let $q\geq2$ be a positive integer, then every real $x\in[0,1]$ has a $q$-adic
representation of the form
\[
  x=\sum_{h=1}^\infty d_h(x)q^{-h}
\]
with $d_h(x)\in\mathcal{D}:=\{0,1,\ldots,q-1\}$. We call a number
$x\in[0,1]$ {\it normal} with respect to the base $q$ if for any $k\geq1$
and any block $\mathbf{b}=b_1\ldots b_k$ of $k$ digits the frequency
of occurrences of this block tends to the expected one, namely
$q^{-k}$. In particular, let $N_n(\mathbf{b},x)$ be the number of
occurrences of $\mathbf{b}$ among the first $n$ digits, \textit{i.e.}
\[
N_n(\mathbf{b},x)=\#\left\{0\leq h<n:d_{h+1}(x)=b_1,\ldots,d_{h+k}(x)=b_k\right\}.
\]
Then we call $x\in[0,1]$ {\it normal of order $k$} in base $q$ if for every
block $\mathbf{b}$ of length $k$ we have
\[
\lim_{n\to\infty}\frac{N_n(\mathbf{b},x)}{n}=q^{-k}.
\]
Furthermore, we call a number 
%normal if it is normal of order $k$ for every $k\geq1$ and 
{\it absolutely normal} if it is normal in every base
$q\geq2$.

In 1909 Borel~\cite{borel1909:les_probabilites_denombrables} showed
that Lebesgue almost every real is absolutely normal. This motivated
people to look for a concrete example of such a number. It took
more than 20 years until 1933 when
Champernowne~\cite{Champernowne1933:construction_decimals_normal}
provided the first explicit construction by showing that the number
\[
0.1\,2\,3\,4\,5\,6\,7\,8\,9\,10\,11\,12\,13\,14\,15\,16\,\ldots
\]
is normal to base $10$. This construction was generalized to different
sequences (such as primes, polynomials etc.) and different numeration
systems (such as $\beta$-expansion, canonical number systems etc.).

In particular, polynomials and polynomial like sequences were
considered. Besicovitch
\cite{besicovitch1935:asymptotic_distribution_numerals} investigated
the sequence of squares. This was extended by Davenport and Erd{\H o}s
\cite{davenport_erdoes1952:note_on_normal} to polynomials with integer
coefficients. Schiffer \cite{schiffer1986:discrepancy_normal_numbers}
further extended this to polynomials with rational integers. Finally
Nakai and Shiokawa
\cite{nakai_shiokawa1992:discrepancy_estimates_class,
  nakai_shiokawa1990:class_normal_numbers} used real polynomials and
pseudo-polynomials, respectively. In parallel polynomial sequences
over the primes were used. Copeland and Erd{\H o}s
\cite{copeland_erdoes1946:note_on_normal} started by using the
sequence of primes, Nakai and Shiokawa
\cite{nakai_shiokawa1997:normality_numbers_generated} used integer
polynomials evaluated at the primes and finally Madritsch
\cite{madritschpear:construction_normal_numbers} used
pseudo-polynomials evaluated at the primes.

On the other hand the Champernowne construction was extended to
different underlying dynamical systems. Normal sequences for Bernoulli
shifts and continued fractions were already investigated by Postnikov
and
Pyatecki{\u\i}~\cite{postnikov_pyateckii1957:markov_sequence_symbols,
  postnikov_pyateckii1957:bernoulli_normal_sequences}, see also
Postnikov \cite{postnikov1960:arithmetic_modeling_random}. Furthermore
normal sequences for Markov shifts and intrinsically ergodic subshifts
were constructed by Smorodinsky and Weiss
\cite{smorodinsky_weiss1987:normal_sequences_markov}. A different
construction for continued fractions is due to Adler \textit{et
  al.}~\cite{adler_keane_smorodinsky1981:construction_normal_number}. Generalizations
to $\beta$-shifts are due to Bertrand-Mathis and
Volkmann~\cite{bertrand-mathis_volkmann1989:epsilon_k_normal} and Ito
and Shiokawa~\cite{ito_shiokawa1975:construction_eta_normal}. However,
the lather construct only a sequence with the right frequency of block
, but which is not admissible as an expansion in the
$\beta$-expansion. The normality of the Champernowne number with
respect to numeration systems in the Gaussian integers was
investigated by Dumont \textit{et al.}
\cite{dumont_grabner_thomas1999:distribution_digits_in}. An important
feature of the Champernowne sequence in dynamical systems fulfilling
the specification property is that it is generic for the maximal
measure, which was shown by Bertrand-Mathis
\cite{bertrand-mathis1988:points_generiques_de}.

Most of the above constructions have in common, that they are aiming
for normal numbers or equivalently sequences that are generic for the
maximum measure. Moreover most of these constructions consider the
full shift, which means, that there are no restrictions on the
concatenation of blocks. In the present paper, however, we
build a framework for the construction of normal sequences in a broad
class of dynamical systems. In particular, we want that for any given
shift invariant measure $\mu$ one can take a sequence of words
$\{\mathbf{w}_i\}_{i\geq1}$ together with a sequence of repetitions $(\ell_i)_{i\geq1}$, satisfying some growing condition (which we
will call $\mu$-good) in order to construct a $\mu$-generic word.

As an example for such a sequence of words we will modify the
Champernowne construction such that we get arbitrarily close to any
given shift invariant measure in a dynamical system fulfilling the
specification property. This is motivated by recent constructions by
Altomare and Mance
\cite{altomare_mance2011:cantor_series_constructions} and Mance
\cite{mance2011:construction_normal_numbers,
  mance2012:cantor_series_constructions}. 

\section{Definitions and statement of results}
Our basis is a symbolic dynamical system that fulfills the
specification property, which we will define in the sequel. In our
definitions we mainly follow the articles of Bertrand-Mathis and
Volkmann \cite{bertrand-mathis1988:points_generiques_de,
  bertrand-mathis_volkmann1989:epsilon_k_normal} as well as the book
of Lind and Marcus
\cite{lind_marcus1995:introduction_to_symbolic}. Let $A$ be a fixed
(possibly infinite) alphabet. We denote by $A^+$ the semigroup
generated by $A$ under catenation. Let $\varepsilon$ denote the empty
word and $A^*=A^+\cup\{\varepsilon\}$. The length of a word
$\omega=a_1a_2\ldots a_k$ with $a_i\in A$ for $1\leq i\leq k$ is
denoted by $\lvert\omega\rvert=k$ and we write $A^k$ for the set of
words of length $k$ (over $A$). Furthermore let $A^\N$ be the set of
infinite words (over $A$).

A set $\mathcal{L}\subset A^*$ is called a {\it language}.
%  Furthermore we
% call a language $X$ a code if $X\subset A^+$ and if each word
% $\omega\in X$ has a unique factorisation
% $$\omega=\omega_1\omega_2\cdots\omega_\ell,\quad \omega_i\in X.$$
% The set $W=W(X^*)$ of the factors of all words $\omega\in X^*$ is
% called the language generated by the code $X$.
% For any language
% $\mathcal{L}\subset A^*$ we split the set of factors up according to
% their length as follows
% $$
% W=W(\mathcal{L}^*)=\bigcup_{n=1}^\infty\mathcal{L}_n,
% \quad\mathcal{L}_n=\{\omega\in
% W(\mathcal{L}^*)\colon\lvert\omega\rvert=n\}.$$
We say that a language $\mathcal{L}$ fulfills the {\it specification
  property} if there exists a positive integer $j$ such that for any
two words $\mathbf{a},\mathbf{b}\in \mathcal{L}$ there exists a word
$\mathbf{u} \in \mathcal{L}$ with $\lvert \mathbf{u}\rvert\leq j$ such
that $\mathbf{a} \mathbf{u} \mathbf{b} \in \mathcal{L}$. Informally
speaking that means that we can ``glue'' together any two words in
that language by a finite amount of glue. For any pair of finite words
$\mathbf{a}$ and $\mathbf{b}$ we fix a
$\mathbf{u}_{\mathbf{a},\mathbf{b}}$ with
$\lvert\mathbf{u}_{\mathbf{a},\mathbf{b}}\rvert\leq j$ such that
$\mathbf{a}\mathbf{u}_{\mathbf{a},\mathbf{b}}\mathbf{b}\in\mathcal{L}$. Then
for $\mathbf{a},\mathbf{a}_1,\ldots,\mathbf{a}_m\in\mathcal{L}$ and
$n\in\N$ we write
$$\mathbf{a}_1\odot\mathbf{a}_2\odot\cdots\odot\mathbf{a}_m
:=\mathbf{a}_1\,\mathbf{u}_{\mathbf{a}_1,\mathbf{a}_2}\,\mathbf{a}_2\,
\mathbf{u}_{\mathbf{a}_2,\mathbf{a}_3}\,\mathbf{a}_3\cdots
\,\mathbf{a}_{m-1}\,\mathbf{u}_{\mathbf{a}_{m-1},\mathbf{a}_m}\,\mathbf{a}_m$$
and recursively define
\[\mathbf{a}^{\odot 1}=\mathbf{a}\quad\text{and}\quad
\mathbf{a}^{\odot n}=\mathbf{a}\odot\mathbf{a}^{\odot(n-1)}
\quad\text{for }n\geq2.\]
For a language $\mathcal{L} \subset A^*$ let
$W^\infty=W^\infty(\mathcal{L})$ be the set of infinite words
generated by $\mathcal{L}$, \textit{i.e.} the set of sequences
$\omega=(a_i)_{i\geq1}$ with $a_ia_{i+1}\cdots a_k\in \mathcal{L}$ for
any $1\leq i<k<\infty$.

We introduce the discrete topology on $A$ and the corresponding
product topology on $A^\N$.  Let $\omega=(a_i)_{i\geq1}\in A^\N$, then
we define the shift operator $T$ as the mapping
$(T(\omega))_i=a_{i+1}$ for $i\geq1$. We associate with each language
$\mathcal{L}$ the symbolic dynamical system
$$S_\mathcal{L}=S=(W^\infty,\mathfrak{B},T,I),$$
where $W^\infty=W^\infty(\mathcal{L})$; $\mathfrak{B}$ is the
$\sigma$-algebra generated by all cylinder sets of $A^\N$,
\textit{i.e.} sets of the form
$$c(\mathbf{w})=[\mathbf{w}]=\{a_1a_2a_3\ldots\in A^\N\colon a_1a_2\ldots
a_n=\mathbf{w}\}$$ for some word $\mathbf{w}\in A^n$ of length $n$;
$T$ is the shift operator; and $I$ is the set of all $T$-invariant
probability measures $\mu$ on $\mathfrak{B}$. We will also write
$\mu(\mathbf{w})$ for $\mu(c(\mathbf{w}))$. Note that $W^\infty$ is
invariant under $T$, \textit{i.e.} $\forall\omega\in W^\infty\colon
T\omega\in W^\infty$, and closed with respect to this topology.

% With each symbolic dynamical system $S$ we associate the entropy
% $$h(W^\infty)=\sup_{\mu\in I}h(\mu),$$
% where $h(\mu)$ denotes the entropy\footnote{For a definition see
%   chapter 2 of Billingsley \cite{billingsley1965:ergodic_theory_and}}
% of the measure $\mu$. For finite alphabets it is known (\textit{cf.}
% Proposition 19.13 of Denker \textit{et al.}
% \cite{denker_grillenberger_sigmund1976:ergodic_theory_on}) that there
% always exists a unique measure $\chi_\mathcal{L}\in I$, called {\it measure
% of maximal entropy} or {\it equilibrium state}, such that
% $h(W^\infty)=h(\chi)$. Bertrand-Mathis
% \cite{bertrand-mathis1988:points_generiques_de} has shown, that this
% measure can be generated by a Champernowne type construction.

Now we fix a $T$-invariant measure $\mu\in I$. A word $\mathbf{b}\in
W(\mathcal{L}^*)$ is {\it $\mu$-admissible} if $\mu(\mathbf{b}) \neq 0$.  Let
$\admu$ denote the set of $\mu$-admissible words and let $\adkmu$
denote the set of $\mu$-admissible words of length $k$.  Given words
$\mathbf{b}\in A^k$ and $\omega=a_1a_2a_3\ldots\in A^\N$ we will let $N_n(\mathbf{b},\omega)$ denote
the number of times the word $\mathbf{b}$ occurs starting in position
no greater than $n$ in the word $\omega$, \textit{i.e.}
$$N_n(\mathbf{b},\omega)=\#\{0\leq i<n\colon a_{i+1}a_{i+2}\cdots a_{i+k}=\mathbf{b}\}.$$
If $\mathbf{w}$ is finite we will often write $N(\mathbf{b},\mathbf{w})$ in
place of $N_{|\mathbf{w}|-\lvert\mathbf{b}\rvert+1}(\mathbf{b},\mathbf{w})$.
%A weighting $\mu$ if {\it finitely supported} is $\admu$ is finite. 
%A sequence of measures $(\mu_i)_{i=1}^{\infty} \in \left(\MNN\right)^{\Nc{0}}$ {\it converges} to $\mu \in \MNN$ (written $\mu_i \to \mu$) if $\mathscr{D}_{\mu_i} \subset \admu$,

%$\mu_i(\mathbf{b})$ is eventually non-increasing, and $\mu_i(\mathbf{b}) \to \mu(\mathbf{b})$ for all blocks $B$.
%check if we need admissability

The following definition is a generalization of the concept of
$(\varepsilon,k)$-normality originally due to Besicovitch
\cite{besicovitch1935:asymptotic_distribution_numerals}. 
\begin{defi}
  Suppose that $0<\epsilon < 1$, $k$ is a positive integer and $\mu
  \in I$.  A word $\mathbf{w}$ is called {\it
    $(\epsilon,k,\mu)$-normal} if for all $t \leq k$ and words
  $\mathbf{b}$ in $\mathscr{D}_{\mu,t}$, we have
  \[
  \mu(\mathbf{b})|\mathbf{w}|(1-\epsilon) \leq N(\mathbf{b},\mathbf{w})
  \leq \mu(\mathbf{b})|\mathbf{w}|(1+\epsilon).
  \]

  An infinite word $\omega\in A^\N$ is called {\it $\mu$-normal of
    order $k$} if for every admissible word $\mathbf{b}$ of length $k$
  we have
  \[
  \lim_{n\to\infty}\frac{N_n(\mathbf{b},\omega)}n=\mu(\mathbf{b}).
  \]
  We denote by $\Nmuk{\mu}{k}$ the set of all $\mu$-normal words of
  order $k$. Furthermore we call $\omega$ {\it $\mu$-normal} (or
  equivalently {\it generic for $\mu$}) if $\omega \in \Nm:=
  \bigcap_{k=1}^\infty \Nmuk{\mu}{k}$.
  % Similar to above we call a number $\mu$-normal if it is
  % $\mu$-normal of order $k$ for every $k\geq1$.
\end{defi}

Let $(k_i)_{i\geq1}$ be a sequence of positive integers. For $i\geq1$
let $\nu_i\colon A^{k_i}\to[0,1]$ be a finite-length shift-invariant
probability. Then we call $(\nu_i)_{i\geq1}$ an approximation scheme
for $\mu$, if $(\nu_i)_{i\geq1}$ converges weakly to $\mu\in I$
(written $\nu_i \to \mu$). Here we silently make the additional
assumptions that $\mathscr{D}_{\nu_i} \subset \admu$ \footnote{A
  version of our main theorem is still true if we drop the condition
  $\mathscr{D}_{\nu_i} \subset \admu$, but every example we will
  consider has this property.}  for all $i\geq1$ and that
$\nu_i(\mathbf{b})$ is eventually non-increasing in $i$.

Furthermore let $(\mathbf{w}_i)_{i\geq1}$ be a sequence of finite
words and $(\ell_i)_{i\geq1}$ be a non-decreasing sequence of positive
integers. Then we call $(\mathbf{w}_i,\ell_i)_{i\geq1}$ $\mu$-good
with respect to the approximation scheme $(\nu_i)_{i\geq1}$ if each
$\mathbf{w}_i$ is $(\varepsilon_i,k_i,\nu_i)$-normal satisfying
\begin{equation}\labeq{good1}
\frac {1} {\eim-\ei}=o( \lxi );
\end{equation}
\begin{equation}\labeq{good2}
\frac {\limo} {\li} \cdot \frac { \lxim} { \lxi }=o(i^{-1});
\end{equation}
\begin{equation}\labeq{good3}
\frac {1} {\li} \cdot \frac { \lxip} { \lxi }=o(1).
\end{equation}

Now we are able to state our main theorem.

\begin{mthm}\labmt{1} Let $(k_i)_{i\geq1}$ be a sequence of positive
  integers and let $\nu_i\colon A^{k_i}\to[0,1]$ be a finite-length
  shift-invariant probability, which is an approximation scheme for
  $\mu$. Furthermore let $(\mathbf{w}_i)_{i\geq1}$ be a sequence of
  finite words and $(\ell_i)_{i\geq1}$ be a non-decreasing sequence of
  positive integers. Suppose that $(\mathbf{w}_i,\ell_i)_{i\geq1}$ is
  $\mu$-good with respect to $(\nu_i)_{i\geq1}$, then for each
  $k\in[1,\limsup_{i\to\infty}k_i]\cap\N$, the infinite word
  $\omega=\mathbf{w}_1^{\odot \ell_1}\odot\mathbf{w}_2^{\odot
    \ell_2}\odot\cdots$ is $\mu$-normal of order $k$. Moreover, if
  $\limsup_{i\to\infty}k_i=\infty$, then $\omega$ is $\mu$-normal
\end{mthm}

We postpone the proof of \refmt{1} to Section
\ref{sec:proof_of_main_theorem} and start by presenting our
construction of a $\mu$-good sequence $(\mathbf{w}_i)_{i\geq1}$ in the
following section. In Section \ref{sec:proof_of_main_theorem} we build
up the toolbox and prove \refmt{1}. Finally, we apply our constructed
sequence of blocks $(\mathbf{w}_i)_{i\geq1}$ from Section
\ref{sec:construction} to different number systems. These number systems
are based on different requirements (finite or infinite digit set,
restrictions on the digit set, \textit{etc.}) and we need to adapt our
construction to these specific cases.

\section{The construction}\label{sec:construction}
Our construction is very similar to the Champernowne type construction
of Bertrand-Mathis and Volkmann
\cite{bertrand-mathis_volkmann1989:epsilon_k_normal,
  bertrand-mathis1988:points_generiques_de}. However, our goal is to
construct a normal sequence which is generic for any given shift
invariant measure. As a consequence of this general case our
construction is not so efficient in that it uses many repetitions of
certain blocks. A more careful control of the available words and
their distribution, would lead to a reduction in the number of copies
$\ell_i$ for special cases (\textit{cf.}  Vandehey
\cite{vandehey2013:simpler_normal_number}).

In our construction we have to face two main issues. The first one is
that our digit set might be infinitely large. This we can easily
circumvent by increasing the digit set in every step (\textit{i.e.} in
every $\mathbf{w}_i$). The other issue we have to face is that there
might be restrictions on the concatenation of words. For example, if
we take the golden mean as basis of a $\beta$-expansion, two
successive ones are forbidden in the expansion. However, concatenating
$1001$ and $1010$, which are admissible as such, yields the word
$10011010$, which is not admissible. Thus similar to above we use
the specification property in order to glue the words together.

Let $j$ be the maximum size of the padding given by the specification
property and let
$\mathcal{P}_{b,\womega}=\{\mathbf{p}_1,\ldots,\mathbf{p}_{b^\womega}\}$
be the set of all possible words of length $\womega$ of the alphabet
$A=\{0,1,\ldots,b-1\}$ of digits in base $b$.

Furthermore
$m_k=\min\{\mu(\mathbf{b}):\mathbf{b}\in\mathscr{D}_{\mu,k}\}$ for
$k\geq1$ and $M$ be an arbitrary large constant such that $M \geq
\frac1{m_\womega}$.

The central tool for our construction will be a weighted concatenation of the
words $\tilde{\mathbf{p}}_i$, \textit{i.e.}, 
\[
\mathbf{p}_{b,\womega, M}:= \mathbf{p}_1^{\odot\lc M\mu(\mathbf{p}_1)\rc}\odot \mathbf{p}_2^{\odot \lc M\mu(\mathbf{p}_2)\rc }\odot\cdots
 \odot \mathbf{p}_{b^\womega}^{\odot\lc M\mu(\mathbf{p}_{b^\womega})\rc}.
\]

In the following we show the $(\varepsilon,k)$-normality of
$\mathbf{p}_{b,\womega,M}$ for $k\leq\womega$. Thus it suffices to find
an $\varepsilon$ such that for all words $\mathbf{b}$ of length
$k\leq\womega$ we have
\begin{gather}\labeq{p:epsi:k}
(1-\varepsilon)\mu(\mathbf{b})\leq\frac{N(\mathbf{b},\mathbf{p}_{b,\womega,M})}{\lvert \mathbf{p}_{b,\womega,M}\rvert}
\leq (1+\varepsilon)\mu(\mathbf{b})
\end{gather}
To this end we need lower and upper bounds for the length of
$\mathbf{p}_{b,\womega,M}$ as well as lower and upper bounds for the
number of occurrences of a fixed block $\mathbf{b}$ within
$\mathbf{p}_{b,\womega,M}$.

Starting with the estimation of the length of
$\mathbf{p}_{b,\womega,M,j}$ we get as upper bound
\begin{align}\label{construction:length:upperbound}
\lvert \mathbf{p}_{b,\womega,M}\rvert
\leq\sum_{i=1}^{b^\womega}\lc M\mu(\mathbf{p}_i)\rc(j+\womega)
\leq M(j+\womega)\sum_{i=1}^{b^\womega}\mu(\mathbf{p}_i)+(j+\womega)b^\womega
=(j+\womega)\left(M+b^\womega\right).
\end{align}
On the other hand we obtain as lower bound
\begin{align}\label{construction:length:lowerbound}
\lvert \mathbf{p}_{b,\womega,M}\rvert
\geq\sum_{i=1}^{b^\womega}\lc M\mu(\mathbf{p}_i)\rc \womega
\geq M\womega\sum_{i=1}^{b^\womega}\mu(\mathbf{p}_i)=M\womega.
\end{align}

Now we provide upper and lower bounds for the number of
occurrences of a word $\mathbf{b}$ of length $k$ in $\mathbf{p}_{b,\womega,M}$.

For the lower bound we only count the possible occurrences within a
$\mathbf{p}_i$. If there is an occurrence then we can write $\mathbf{p}_i$ as $\mathbf{c}_1\mathbf{b}\mathbf{c}_2$ with
possible empty $\mathbf{c}_1$ or $\mathbf{c}_2$. Since the word
$\mathbf{b}$ is fixed and all possible words of length $\womega$ occur
in $\mathbf{p}_{b,\womega,M}$, we let
$\mathbf{c}_1$ and $\mathbf{c}_2$ vary over all possible words. Thus
\begin{equation}\label{construction:counting:lowerbound}
\begin{split}
N(\mathbf{b},\mathbf{p}_{b,\womega,M})&\geq \sum_{m=0}^{\womega-k}\sum_{\lvert \mathbf{c}_1\rvert=m}\sum_{\lvert
  \mathbf{c}_2\rvert=\womega-k-m}\lc M\mu(\mathbf{c}_1\mathbf{b}\mathbf{c}_2)\rc\\
&\geq M\sum_{m=0}^{\womega-k}\sum_{\lvert \mathbf{c}_1\rvert=m}\sum_{\lvert
  \mathbf{c}_2\rvert=\womega-k-m}\mu(\mathbf{c}_1\mathbf{b}\mathbf{c}_2)\\
&=M\sum_{m=0}^{\womega-k}\sum_{\lvert \mathbf{c}_1\rvert=m}\sum_{\lvert
  \mathbf{c}_2\rvert=\womega-k-m-1}\sum_{d=0}^{b-1}\mu(\mathbf{c}_1\mathbf{b}\mathbf{c}_2d)\\
% &=\cdots=M\sum_{m=0}^{\womega-k}\sum_{\lvert \mathbf{c}_1\rvert=m-1}\sum_{d=0}^{b-1}\mu(d\mathbf{c}_1\mathbf{b})\\
&=\cdots=M\sum_{m=0}^{\womega-k}\mu(\mathbf{b})=(\womega-k+1)M\mu(\mathbf{b}),
\end{split}
\end{equation}
where we have used the shift invariance of $\mu$, \textit{i.e.} $\sum_{d=0}^{b-1}\mu(d\mathbf{a})=\sum_{d=0}^{b-1}\mu(\mathbf{a}d)=\mu(\mathbf{a})$.

For the upper bound we have to consider
  several different possibilities: The word $\mathbf{b}$ can occour 
  \begin{enumerate}
  \item within $\mathbf{p}_i$,
  %\item within the padding before or after $\mathbf{p}_i$ and $\mathbf{p}_i$,
  \item between two similar words $\mathbf{p}_i\odot \mathbf{p}_i$ or
  \item between two different words $\mathbf{p}_i\odot \mathbf{p}_{i+1}$.
  \end{enumerate}
  
  If the word $\mathbf{b}$ is completely within $\mathbf{p}_i$, then
  we again have that $\mathbf{p}_i=\mathbf{c}_1\mathbf{b}\mathbf{c}_2$
  with possible empty $\mathbf{c}_1$ or $\mathbf{c}_2$. By using
  similar means as above we get that
  \begin{align*}
  \sum_{\mathbf{c}_1,\mathbf{c}_2}\lc M\mu(\mathbf{c}_1\mathbf{b}\mathbf{c}_2)\rc
  \leq \sum_{\mathbf{c}_1,\mathbf{c}_2}\left(M \mu(\mathbf{c}_1\mathbf{b}\mathbf{c}_2)+1\right)
  =\cdots=(\womega-k+1)\left(M\mu(\mathbf{b})+b^{\womega-k}\right),
  \end{align*}

  Now we turn our attention to the number of occurrences between two
  consecutive words. First we assume that these words are equal. Let
  $n=\lvert\mathbf{p}_i\odot\mathbf{p}_i\rvert$ be the length of
  the resulting word. Then
  $\mathbf{p}_i\odot\mathbf{p}_i=\mathbf{c}_1\mathbf{b}\mathbf{c}_2$
  with $\womega-k+1\leq\lvert
  \mathbf{c}_1\rvert\leq n-\womega+k-1$. Thus similar to above we
  get that there are
  \begin{align*}
&\sum_{m=\womega-k+1}^{ n-\womega+k-1}\sum_{\lvert \mathbf{c}_1\rvert=m}
  \sum_{\lvert \mathbf{c}_2\rvert= n-k-m}\lc M\mu(\mathbf{c}_1\mathbf{b}\mathbf{c}_2)\rc\\
&\quad\leq  M\sum_{m=\womega-k+1}^{ n-\womega+k-1}\sum_{\lvert \mathbf{c}_1\rvert=m}
  \sum_{\lvert \mathbf{c}'_2\rvert= n-k-m-1}\sum_{d=0}^{b-1}\mu(\mathbf{c}_1\mathbf{b}\mathbf{c}'_2d)
  +\sum_{m=\womega-k+1}^{ n-\womega+k-1}b^{ n-k}\\
% &\quad=\cdots=M\sum_{m=\womega-k+1}^{ n-\womega+k-1}\sum_{\lvert \mathbf{c}'_1\rvert=m-1}
%   \sum_{d=0}^{b-1}\mu(d\mathbf{c}'_1\mathbf{b})+(k-1)b^{ n-k}\\
&\quad=\cdots=M\sum_{m=\womega-k+1}^{ n-\womega+k-1}\mu(\mathbf{b})+( n-2\womega+k-1)b^{ n-k}\\
&\quad=( n-2\womega+k-1)\left(M\mu(\mathbf{b})+b^{ n-k}\right)\\
&\quad\leq(j+k-1)\left(M\mu(\mathbf{b})+b^{2\womega+j-k}\right)
  \end{align*}
occurrences between two identical words.

Finally, we trivially estimate the number of occurrences between two
different words by their total amount, which is
$\leq(j+k-1)b^\womega$.

Combining these three bounds and using $k\leq\womega$ we get as upper
bound for the number of occurrences
\begin{equation}
\label{construction:counting:upperbound}
\begin{split}
&N(\mathbf{b},\mathbf{p}_{b,\womega,M})\\
% &\leq
% \sum_{m=0}^{\womega-k}\sum_{\lvert \mathbf{c}_1\rvert=m-1}\lc M\mu(\mathbf{c}_1\mathbf{b}\mathbf{c}_2)\rc
% +\sum_{m=\womega-k+1}^{\womega-1}\sum_{\lvert \mathbf{c}_1\rvert=m}
%   \sum_{\lvert \mathbf{c}_2\rvert=2\womega-k-m}\lc M\mu(\mathbf{c}_1\mathbf{b}\mathbf{c}_2)\rc+(k-1)b^\womega\\
&\quad\leq(\womega-k+1)\left(M\mu(\mathbf{b})+b^{\womega-k}\right)
+(j+k-1)\left(M\mu(\mathbf{b})+b^{2\womega+j-k}\right)
+(j+k-1)b^\womega\\
&\quad\leq(\womega+j)\left(M\mu(\mathbf{b})+b^{2\womega+j-k}\right).
\end{split}
\end{equation}

Now we calculate $\varepsilon$ such that \refeq{p:epsi:k} holds. Using
our lower bound for the number of occurrences in
(\ref{construction:counting:lowerbound}) together with our upper bound
for the length in (\ref{construction:length:upperbound}) we get that
\begin{align*}
\frac{N(\mathbf{b},\mathbf{p}_{b,\womega,M})}{\lvert \mathbf{p}_{b,\womega,M}\rvert}
\geq\frac{(\womega-k+1)M\mu(\mathbf{b})}{(\womega+j)\left(M+b^{\womega}\right)}
\geq\mu(\mathbf{b})\left(1-\frac{j+k-1}{\womega+j}\right)\left(1-\frac{b^\womega}{M+b^{\womega}}\right)
\end{align*}
which implies for $\varepsilon$ the upper bound
\[
\varepsilon\leq\frac{j+k-1}{\womega+j}+\frac{b^\womega}{M+b^{\womega}}.
\]

On the other side an application of the upper bound for the number of
occurrences in (\ref{construction:counting:upperbound}) together with
the lower bound for the length in
(\ref{construction:length:lowerbound}) yields
\begin{align*}
\frac{N(\mathbf{b},\mathbf{p}_{b,\womega,M})}{\lvert \mathbf{p}_{b,\womega,M}\rvert}
%\leq\frac{(\womega+j)\left(M\mu(\mathbf{b})+b^\womega\right)+(k-1)b^{2\womega+j-k}}{(\womega+j)M}
\leq\mu(\mathbf{b})\left(1+\frac j\womega\right)\left(1+\frac1{m_k}\frac{b^{2\womega+j-k}}{M}\right).
\end{align*}

Putting these together we get that $\mathbf{p}_{b,\womega,M}$ is $(\varepsilon,
k,\mu)$-normal for 
\begin{equation}\label{construction:eps-k-mu:normality}
k\leq\womega\quad\text{and}
\quad
\varepsilon\leq\max
\left(\frac{j+k-1}{\womega+j}+\frac{b^\womega}{M+b^{\womega}},
\frac j\womega+\frac1{m_k}\frac{b^{2\womega+j-k}}{M}\right).
\end{equation}

\section{Proof of \refmt{1}}\label{sec:proof_of_main_theorem}

\begin{rem}
  In our proof we will use a classical counting argument. We could use
  a variant of the ``hot spot lemma'' (\textit{cf.} Moshchevitin and
  Shkredov
  \cite{moshchevitin_shkredov2003:pyatetskii_shapiro_criterion} and
  Shkredov \cite{shkredov2010:pyatetskii_shapiro_normality}). However,
  on the one hand, since their results are for the full shift over
  finite and infinite alphabets, we need to develop a variant of the
  ``hot spot lemma'' for dynamical systems satisfying the
  specification property. On the other hand, since our proof follows
  along similar lines to the proof of Main Theorem 1.15 in
  \cite{mance2011:construction_normal_numbers}, we will only include
  those parts that differ significantly and omit the proofs, which are
  similar to proofs of lemmas in
  \cite{mance2011:construction_normal_numbers}.
\end{rem}

Throughout this section, we will fix a sequence
$W=((\mathbf{w}_i,\ell_i))_{i=1}^{\infty}$ that is $\mu$-good for the
approximation scheme $(\nu_i)_{i\geq1}$. Suppose that every
$\mathbf{w}_i$ is $(\epsilon_i,k_i,\nu_i)$-normal. Then we define the
set of supported lengths $R(W)=[1,\limsup_{i\to\infty}k_i]\cap\N$.

Set
$\omega=\mathbf{w}_1^{\odot \ell_1}\odot\mathbf{w}_2^{\odot
  \ell_2}\odot\cdots$ to be the constructed infinite word and denote
by $\sigma_k$ the $k$th block, \textit{i.e.}
\[\sigma_k=\mathbf{w}_k^{\odot
  \ell_k}\mathbf{u}_{\mathbf{w}_k,\mathbf{w}_{k+1}}\] be the $k$th
block. Let $L_i$ be the length of the concatenation up to the $i$th
block, \textit{i.e.}
\[L_i=\sum_{k=1}^i\lvert\sigma_k\rvert=\sum_{k=1}^i\left(\ell_k|\mathbf{w}_k|+(\ell_{k}-1)|\mathbf{u}_{\mathbf{w}_k,\mathbf{w}_k}|+|\mathbf{u}_{\mathbf{w}_k,\mathbf{w}_{k+1}}|\right).\]
For a given $n$, the letter
$i=i(n)$ will always be understood to be the positive integer that
satisfies $L_{i} < n \leq L_{i+1}$, \textit{i.e.} position $n$ lies in
the block $\sigma_{i+1}$. Let $m=n-L_i$, then we consider
$\sigma_{i+1}\vert m$. Let $x$ be the largest integer such that
\[\sigma_{i+1}\vert
m=(\mathbf{w}_{i+1}\mathbf{u}_{\mathbf{w}_{i+1},\mathbf{w}_{i+1}})^x\mathbf{v}\]
Then $m$ can be written in the form
$$
m=x(\lxip+\lvert\mathbf{u}_{\mathbf{w}_{i+1},\mathbf{w}_{i+1}}\rvert)+y
$$
with $y=\lvert\mathbf{v}\rvert$. We have that $x$ and $y$ satisfy
\[0 \leq x<\lip \textrm{ and } 0 \leq y <  \lxip+j,\]
where $j$ is the bound from the specification property.

Thus, we can write the first $n$ digits of $\omega$ as concatenation
of the complete blocks $\sigma_1,\ldots,\sigma_i$, the $x$ repetitions
of the word $\mathbf{w}_{i+1}$ and the rest $\mathbf{v}$,
\textit{i.e.}
\[
\omega\vert_n=\mathbf{w}_1^{\odot \ell_1}\odot\mathbf{w}_2^{\odot 
  \ell_2}\odot\cdots\odot\mathbf{w}_{i-1}^{\odot \ell_{i-1}}\odot\mathbf{w}_i^{\odot
  \ell_i}\odot\mathbf{w}_{i+1}^{\odot x}\odot \mathbf{v}.
\]

For a word $\mathbf{b}$, let
\begin{gather}\label{phi_n:of:b}
\phi_n(\mathbf{b})=\sum_{k=1}^i \lvert\sigma_k\rvert\nu_k(\mathbf{b}) + m \nuipb.
\end{gather}
Since $(\mathbf{w}_i,\ell_i)_{i=1}^\infty$ is $\mu$-good, we have that
$\lim_{n \to \infty} \frac {\phi_n(\mathbf{b})} {n} =
\mu(\mathbf{b})$. Therefore $\omega$ is $\mu$-normal if and only if
\begin{gather}\labeq{bnd}
\lim_{n \to \infty} \frac {N_n(\mathbf{b},\omega)} {\phi_n(\mathbf{b})}=1
\end{gather}
for all words $\mathbf{b}\in\admu$.  

For a given word $\mathbf{b}$ of supported length $k\in R(W)$, the
following lemma, which is proved identically to Lemma 2.1 and Lemma
2.2 in \cite{mance2011:construction_normal_numbers}, provides us with
upper and lower bounds for $N_n(\mathbf{b},\omega)$.
\begin{lem}\labl{l2.2}
If $k \leq k_i$ and $\mathbf{b} \in \mathscr{D}_{\nu_i,k}$, then
\begin{align*}
N_n(\mathbf{b},\omega)&\leq
L_{i-1}+(1+\epsilon_i){\nuib} \ell_i  \lxi +(k+j)(\ell_i+1)+((1+\epsilon_{i+1}){\nuipb}  \lxip +k+j)x+y\\
\intertext{and}
N_n(\mathbf{b},\omega)&\geq(1-\epsilon_i){\nuib} \ell_i  \lxi +(1-\epsilon_{i+1}){\nuipb}  \lxip. 
\end{align*}
\end{lem}

Now we estimate $\frac{N_n(\mathbf{b},\omega)}{\phi_n(\mathbf{b})}-1$
from above and below. On the one hand using the upper bound for
$N_n(\mathbf{b},\omega)$ in \refl{l2.2} and the definition of
$\phi_n(\mathbf{b})$ in (\ref{phi_n:of:b}) yields
\begin{multline*}
\frac{N_n(\mathbf{b},\omega)}{\phi_n(\mathbf{b})}-1\\
\leq\frac{L_{i-1}+(\epsilon_i {\nuib} \lxi
    +(k+j))\ell_i +(\epsilon_{i+1}{\nuipb} \lxip +(k+j)) x+y}
  {\phi_{L_i}(\mathbf{b})+{\nuipb}\left(\lvert\mathbf{w}_{i+1}\mathbf{u}_{\mathbf{w}_{i+1},\mathbf{w}_{i+1}}\rvert
      x+y\right)}
=:  g_{i,\mathbf{b}}(x,y)
\end{multline*}
On the other hand we combine the lower bound for
$N_n(\mathbf{b},\omega)$ in \refl{l2.2} and the definition of
$\phi_n(\mathbf{b})$ in (\ref{phi_n:of:b}) gives
\begin{multline*}
\frac{N_n(\mathbf{b},\omega)}{\phi_n(\mathbf{b})}-1\\
\geq-\frac{\phi_{L_{i-1}}(\mathbf{b})+\epsilon_i
    {\nuib} \ell_i \lxi +\nuipb\left(\epsilon_{i+1}\lxip
      +\lvert\mathbf{u}_{\mathbf{w}_{i+1},\mathbf{w}_{i+1}}\rvert\right) x
    +{\nuipb} y}
  {\phi_{L_i}(\mathbf{b})
  +{\nuipb}\left(\lvert\mathbf{w}_{i+1}\mathbf{u}_{\mathbf{w}_{i+1},\mathbf{w}_{i+1}}\rvert x+y\right)}=:- f_{i,\mathbf{b}}(x,y)
\end{multline*}

Therefore
\[
\left| \frac {N_n(\mathbf{b},\omega)} {\phi_n(\mathbf{b})} - 1 \right|
<\max \left( f_{i,\mathbf{b}} (x,y),g_{i,\mathbf{b}} (x,y) \right).
\]
However, since the numerator of $g_{i,\mathbf{b}} (x,y)$ is clearly
greater than the numerator of $f_{i,\mathbf{b}} (x,y)$ and their
denominators are the same we deduce the following
\begin{lem}\labl{l2.4}
For any $i$ let $k \in R(W)$, $k \leq k_i$, and $\mathbf{b} \in \mathscr{D}_{\nu_i,k}$.  Then
\begin{equation}
\left| \frac {N_n(\mathbf{b},\omega)} {\phi_n(\mathbf{b})} -1 \right| < g_{i,\mathbf{b}} (x,y).
\end{equation}
\end{lem}

We are looking for a good bound for $g_{i,\mathbf{b}}(x,y)$ where
$(x,y)$ ranges over values in $\{0,1,\ldots,\ell_{i+1} \} \times
\{0,1,\ldots, \lxip -1 \}$. 

% $$
% \twbi=\sup\left(0, \sup \{t : \nuipb \geq \nu_t(\mathbf{b})\}\right).
% $$
% Note that $\twbi<\infty$ as $(\nu_i(\mathbf{b}))_i$ is eventually
% non-increasing. As we go through the construction different parts have
% different weights (different $\nu_i$). The idea is to replace all
% weights by $\nu_{i+1}$. Therefore we look for the position (if it
% exists) where the difference between these to accumulated measures is
% maximal. Thus we set
% $$
% \ewbi=\max(0,L_{\twbi}\nuipb-\phi_{L_i}).
% $$
% This implies that
% \begin{equation}\labeq{ewbi}
% \nuipb L_{i-1} \leq \phi_{L_{i-1}}(\mathbf{b})+\ewbi.
% \end{equation}
 
\begin{lem}\labl{l2.5}
  If $k \in R(W)$, $\ei<1/2$, $\ell_i>0$, $\mathbf{b} \in
  \mathscr{D}_{\nu_i,k}$,
  \begin{equation*}
\lxi >2 \cdot (k+j) +2\frac {L_{i-1}\nuipb-\phi_{L_{i-1}}} {\ell_i \nuib },\quad  \lxip >\frac
  {k+j } {\nuipb(\epsilon_i-\epsilon_{i+1})},
\end{equation*}
  and
\begin{equation}
(x,y) \in \{0,1,\ldots,\ell_{i+1} \} \times \{0,1,\ldots, \lxip +j-1 \},
\end{equation}
then
\begin{equation}
  g_{i,\mathbf{b}}(x,y) < g_{i,\mathbf{b}} (0, \lxip+j)=\frac {(L_{i-1}+\epsilon_i {\nuib}\ell_i \lxi +(k+j)\ell_i) + \lxip +j} {\phi_{L_i}(\mathbf{b})+{\nuipb}\left(\lxip+j\right)}.
\end{equation}
\end{lem}

\begin{proof} 
We note that $g_{i,\mathbf{b}} (x,y)$ is a rational function of $x$ and $y$ of the form
$$
g_{i,\mathbf{b}} (x,y)=\frac {C+Dx+Ey} {F+Gx+Hy}
$$
where
\begin{align*}
C&=L_{i-1}+\epsilon_i {\nuib} \ell_i  \lxi +(k+j)\ell_i, &
D&=\epsilon_{i+1}{\nuipb}  \lxip +(k+j), & 
E&=1,\\
F&=\phi_{L_i}(\mathbf{b}), & 
G&={\nuipb}\lvert\mathbf{w}_{i+1}\mathbf{u}_{\mathbf{w}_{i+1},\mathbf{w}_{i+1}}\rvert,\text{ and} & 
H&={\nuipb}.
\end{align*}
We will show that if we fix $y$, then $g_{i,\mathbf{b}}(x,y)$ is a decreasing function
of $x$ and if we fix $x$, then $g_{i,\mathbf{b}}(x,y)$ is an increasing function of $y$.
To see this, we compute the partial derivatives: 
\begin{equation}\label{partial:derivatives:labeling}\begin{split}
&\frac {\partial g_{i,\mathbf{b}}} {\partial x} (x,y)=\frac {D(F+Gx+Hy)-G(C+Dx+Ey)} {(F+Gx+Hy)^2}=\frac {D(F+Hy)-G(C+Ey)} {(F+Gx+Hy)^2};\\
&\frac {\partial g_{i,\mathbf{b}}} {\partial y} (x,y)=\frac {E(F+Gx+Hy)-H(C+Dx+Ey)} {(F+Gx+Hy)^2}=\frac {E(F+Gx)-H(C+Dx)} {(F+Gx+Hy)^2}.
\end{split}\end{equation}
Thus, the sign of $\frac {\partial g_{i,\mathbf{b}}} {\partial x}
(x,y)$ does not depend on $x$ and the sign of $\frac {\partial
  g_{i,\mathbf{b}}} {\partial y} (x,y)$ does not depend on $y$.  We
will first show that $g_{i,\mathbf{b}} (x,y)$ is an increasing
function of $y$ by verifying that
\begin{equation}\labeq{251}
E(F+Gx)>H(C+Dx).
\end{equation}
Let $\phi_i^*(\mathbf{b})=H\,C={\nuipb}\left(L_{i-1}+\epsilon_i
  {\nuib} \ell_i \lxi +(k+j)\ell_i\right)$. Then Equation \refeq{251} can be
written as
\begin{equation}\labeq{254}
\phi_{L_i}(\mathbf{b}) +  \Bigg[ {\nuipb} \lxip x \Bigg] > \phi_i^*(\mathbf{b}) + \Bigg[ {\nuipb}(\epsilon_{i+1} {\nuipb} \lxip  + (k+j) )x \Bigg].
\end{equation}
We will verify this inequality in two steps by showing
\begin{align*}
\phi_{L_i}(\mathbf{b})> \phi_i^*(\mathbf{b}) \quad\text{and}\quad
\nuipb \lxip x > \nuipb(\epsilon_{i+1} {\nuipb} \lxip  + (k+j) )x.
\end{align*}

In order to show that $\phi_{L_i}(\mathbf{b}) > \phi_i^*(\mathbf{b})$,
we first note that
\[
 \phi_{L_i}(\mathbf{b})
=\phi_{L_{i-1}}(\mathbf{b})+{\nuib}\left((\ell_i-1)\lvert\mathbf{w}_i\mathbf{u}_{\mathbf{w}_i,\mathbf{w}_i}\rvert+\lvert\mathbf{w}_i\mathbf{u}_{\mathbf{w}_i,\mathbf{w}_{i+1}}\rvert\right)
\geq\phi_{L_{i-1}}(\mathbf{b})+{\nuib}\ell_i\lxi.
\]
Thus we need to show that
\begin{equation}\labeq{252} 
\phi_{L_{i-1}}+\nuib\ell_i  \lxi  > {\nuipb} (\epsilon_i {\nuib}\ell_i  \lxi +(k+j)\ell_i)+L_{i-1}\nuipb.
\end{equation}
However, by rearranging terms, \refeq{252} is equivalent to
\begin{equation}\labeq{253} 
 \lxi  > \frac {\nuipb} {\nuib}  \cdot \frac {k+j} {1-{\nuipb}\epsilon_i} +\frac {L_{i-1}\nuipb-\phi_{L_{i-1}}} {\ell_i \nuib (1-\nuipb \ei)}.
\end{equation}
Since $\epsilon_i < 1/2$, we know that $(1-{\nuipb}\epsilon_i)^{-1} <
2$.  Additionally, since $\nuib$ is eventually non-increasing we have
for sufficiently large $i$ that $\nuipb \leq \nuib$. Therefore,
\begin{equation*}
\frac {\nuipb} {\nuib} \cdot \frac {k+j} {1-{\nuipb}\epsilon_i} +\frac
{L_{i-1}\nuipb-\phi_{L_{i-1}}} {\ell_i \nuib (1-\nuipb \ei)}
< 2 \cdot (k+j) +2\frac {L_{i-1}\nuipb-\phi_{L_{i-1}}} {\ell_i \nuib }.
\end{equation*}
But we supposed $ \lxi >2 \cdot (k+j) +2\frac
{L_{i-1}\nuipb-\phi_{L_{i-1}}} {\ell_i \nuib }.$. So \refeq{253} is
satisfied and thus $\phi_{L_i}(\mathbf{b}) > \phi_i^*(\mathbf{b})$.

The last step, for verifying \refeq{254}, is to show that
$$
{\nuipb} \lxip x \geq {\nuipb}(\epsilon_{i+1} {\nuipb} \lxip +(k+j)) x.
$$
However, this is equivalent to
\begin{equation}\labeq{255} 
 \lxip x \geq (\epsilon_{i+1} {\nuipb} \lxip +(k+j)) x.
\end{equation}
Clearly, \refeq{255} is true if $x=0$.  If $x>0$ we can rewrite \refeq{255} as
$$
 \lxip  \geq \frac {1} {1-{\nuipb}\epsilon_{i+1}} \cdot (k+j).
$$
Similar to \refeq{253}, $(1-{\nuipb}\epsilon_{i+1})^{-1} (k+j) \leq 2 (k+j) <  \lxi  \leq  \lxip $.  Thus \refeq{251} is satisfied and $g_{i,\mathbf{b}}(x,y)$ is an increasing function of $y$.

It will be more difficult to show that $\frac {\partial
  g_{i,\mathbf{b}}} {\partial x} (x,y)<0$ in a similar manner so we
proceed as follows: because the sign of $\frac {\partial
  g_{i,\mathbf{b}}} {\partial x} (x,y)$ does not depend on $x$, we
will know that $g_{i,\mathbf{b}} (x,y)$ is decreasing in $x$ if for
each $y$
$$
\lim_{x \rightarrow \infty} g_{i,\mathbf{b}} (x,y)<g_{i,\mathbf{b}}(0,y).
$$
Since $g_{i,\mathbf{b}}(x,y)$ is an increasing function of $y$, we
know for all $y$ that $g_{i,\mathbf{b}}(0,0)<g_{i,\mathbf{b}}(0,y)$.
Hence, it is enough to show that
$$
\lim_{x \rightarrow \infty} g_{i,\mathbf{b}} (x,y) < g_{i,\mathbf{b}}(0,0).
$$
Since $\lim_{x \rightarrow \infty} g_{i,\mathbf{b}} (x,y)=D/G$ and
$g_{i,\mathbf{b}}(0,0)=C/F$, it is sufficient to show that
$CG>DF$, where $C$, $D$, $F$ and $G$ are as in \eqref{partial:derivatives:labeling}. Therefore we have
\begin{equation*}
\begin{split}
  &\left( L_{i-1}+\epsilon_i {\nuib} \ell_i  \lxi +(k+j)\ell_i \right) {\nuipb}\lvert\mathbf{w}_{i+1}\mathbf{u}_{\mathbf{w}_{i+1},\mathbf{w}_{i+1}}\rvert   
  > \left( \epsilon_{i+1}{\nuipb}  \lxip +(k+j) \right) \phi_{L_i}(\mathbf{b})\\
&\quad=\left( \epsilon_{i+1}{\nuipb}  \lxip +(k+j) \right) (\phi_{L_{i-1}}(\mathbf{b})+{\nuib}\left((\ell_i-1)\lvert\mathbf{w}_i\mathbf{u}_{\mathbf{w}_i,\mathbf{w}_i}\rvert+\lvert\mathbf{w}_i\mathbf{u}_{\mathbf{w}_i,\mathbf{w}_{i+1}}\rvert\right)\\
\end{split}
\end{equation*}
Since $0\leq\lvert\mathbf{u}_{\mathbf{a},\mathbf{b}}\rvert\leq j$ it
suffices to show that
% \begin{multline*}
%   \left( L_{i-1}+\epsilon_i {\nuib} \ell_i  \lxi +(k+j)\ell_i \right) {\nuipb}\lvert\mathbf{w}_{i+1}\rvert   \\
% >\left( \epsilon_{i+1}{\nuipb}  \lxip +(k+j) \right) (\phi_{L_{i-1}}(\mathbf{b})+{\nuib}\ell_i\left(\lvert\mathbf{w}_i\rvert+j\right),
% \end{multline*}
% which simplifies to
\begin{multline*}
L_{i-1} {\nuipb}  \lxip  + \epsilon_i {\nuib} {\nuipb} \ell_i  \lxi   \lxip  + (k+j) {\nuipb}\ell_i   \lxip \\
>\left( \epsilon_{i+1}{\nuipb}  \lxip +(k+j) \right) \phi_{L_{i-1}}(\mathbf{b}) + \left( \epsilon_{i+1}{\nuipb}  \lxip +(k+j) \right){\nuib}\ell_i\left(\lxi+j\right) .
\end{multline*}

Similar to above we will verify this in two steps:
\begin{equation}\labeq{mani:1}
\begin{split}
&L_{i-1} {\nuipb}  \lxip >\left( \epsilon_{i+1}{\nuipb}  \lxip +(k+j) \right) \phi_{L_{i-1}}(\mathbf{b}) \textrm{ \ and}\\
&
\epsilon_i {\nuipb}  \lxip  > \epsilon_{i+1}{\nuipb}  \lxip +(k+j),
\end{split}
\end{equation}
Since $L_{i-1}>\phi_{L_{i-1}}(\mathbf{b})$, in order to prove the first inequality of \refeq{mani:1}, it is enough to show that
$$
{\nuipb}  \lxip >\epsilon_{i+1}{\nuipb}  \lxip +(k+j),
$$
which is equivalent to
$$
 \lxip  > \frac {k+j } {\nuipb(1-\epsilon_{i+1}) }.
$$
But $\epsilon_i<1/2$, so 
$$
\frac {k+j } {\nuipb(1-\epsilon_{i+1}) }<\frac {k+j} {\nuipb(\epsilon_i-\epsilon_{i+1}) }< \lxip .
$$

To verify the second inequality of \refeq{mani:1} we note that this is equivalent to  
\[
 \lxip  > \frac {k+j} {\nuipb(\epsilon_i-\epsilon_{i+1})},
\]
which is given in the hypotheses.

So, we may conclude that $g_{i,\mathbf{b}} (x,y)$ is a decreasing
function of $x$ and an increasing function of $y$. Since $x\geq0$ and
$y<\lxip+j$, we achive the given upper bound by setting $x=0$ and
$y=\lxip+j$.
\end{proof}

Set
\[
\epsilon_i'
=g_{i,\mathbf{b}}(x,y) < g_{i,\mathbf{b}} (0, \lxip+j)=\frac {(L_{i-1}+\epsilon_i {\nuib}\ell_i \lxi +(k+j)\ell_i) + \lxip +j} {\phi_{L_i}(\mathbf{b})+{\nuipb}\left(\lxip+j\right)}.
\]
Thus, under the conditions of \refl{l2.4} and \refl{l2.5},
\begin{equation}\labeq{726}
\left| \frac {N_n(\mathbf{b},\omega)} {\phi_n(\mathbf{b})} -1 \right| < \epsilon_i'
\end{equation}

The proof of the following lemma is essentially identicaly to the
combined proofs of Lemma 2.6, Lemma 2.7, and Lemma 2.8 in
\cite{mance2011:construction_normal_numbers} so the proof has been
omitted.

\begin{lem}\labl{l2.8}
If $k\in R(W)$, then $\lim_{i \rightarrow \infty} \epsilon_i'=0.$
\end{lem}

\begin{proof}[Proof of \refmt{1}]
%\begin{proof} 
Let $\mathbf{b} \in \adkmu$ for $k \in R(W)$.  Since $\frac {1}
{\epsilon_{i-1}-\epsilon_i}=o(\lxi)$, there exists $n$ large enough so that $
\lxi $ and $ \lxip $ satisfy the hypotheses of \refl{l2.5}.

Since $\lim_{n \rightarrow \infty} i(n)=\infty$, we conclude by applying \refl{l2.8} in \refeq{726} that
$$
\lim_{n \rightarrow \infty} \left| \frac {N_n(\mathbf{b},\omega)} {\phi_n(\mathbf{b})} -1 \right|=0
$$
which implies that
% $$
% \lim_{n \rightarrow \infty} \frac {N_n(\mathbf{b},\omega)} {\phi_n(\mathbf{b})} =1.
% $$
% Thus, 
$$
\lim_{n \to \infty} \frac {N_n(\mathbf{b},\omega)} {n}=\mu(\mathbf{b}).
$$

On the contrary let $\mathbf{b}\in A^k\setminus\adkmu$. Since
\begin{align*}
1&=\lim_{n\to\infty}\sum_{\mathbf{b}'\in
  A^k}\frac{N_n(\mathbf{b}',\omega)}n\\
&=\sum_{\mathbf{b}'\in \adkmu}
\lim_{n \to \infty} \frac {N_n(\mathbf{b}',\omega)}{n} +
\sum_{\mathbf{b}'\in A^k\setminus\adkmu} \lim_{n \to \infty} \frac
{N_n(\mathbf{b}',\omega)}{n}\\
&=\sum_{\mathbf{b}' \in
  \adkmu}\mu(\mathbf{b}')+ \sum_{\mathbf{b}'\in A^k\setminus\adkmu} \lim_{n \to
  \infty} \frac {N_n(\mathbf{b}',\omega)}{n}\\
&=1+
\sum_{\mathbf{b}'\in A^k\setminus\adkmu} \lim_{n \to \infty} \frac
{N_n(\mathbf{b}',\omega)}{n}
\end{align*}
and $N_n(\mathbf{b}',\omega)\geq0$ we get that
\[\lim_{n \to \infty} \frac {N_n(\mathbf{b},\omega)} {n}=0=\mu(\mathbf{b}).\]

Therefore combining the two limits from above we get for $\mathbf{b}\in A^k$ that
\[\lim_{n \to \infty} \frac {N_n(\mathbf{b},\omega)} {n}=\mu(\mathbf{b}),\] which implies that $\omega \in \Nmuk{\mu}{k}$.% is $\mu$-normal of order $k$.
\end{proof}

\section{Applications}\label{sec:examples}
In the following subsections we show different numeration systems in
which our construction provides normal numbers. In particular, we
consider the $q$-ary expansions, L{\"u}roth series expansion,
$\beta$-expansions and continued fraction expansion. We only have
restrictions on the concatenation in the case of $\beta$-expansions;
all other examples are in the full-shift. It is easy to combine our
construction for $\beta$-expansions and continued fractions in order
to get constructions for $\alpha$-continued fractions (\textit{cf.}
Nakada \cite{nakada1981:metrical_theory_class}) or Rosen-continued
fractions \cite{rosen1954:class_continued_fractions}, which have an
infinite digit set with restrictions on the concatenation of
words. For the relation of normal numbers with respect to different
continued fraction expansion, we refer the interested reader to the
paper of Kraaikamp and Nakada
\cite{kraaikamp_nakada2000:normal_numbers_continued}.

The main ingredient in all our constructions is the following lemma
which follows immediately from our construction in Section
\ref{sec:construction} and \refmt{1}.
\begin{lem}\labl{mainlemma}
  Let $\mu$ be a shift-invariant probability measure and let
  $(\nu_i)_{i\geq1}$ be an approximation scheme for $\mu$.
  Suppose that $q_i \geq 2$, $M_i$ and $\ell_i$ are sequences of
  positive integers such that
  \begin{gather}\label{eqn:good4}
  M_i \geq \left(\min\{\mu(\mathbf{b}):\mathbf{b}\in\mathscr{D}_{\nu_i,i}\}\right)^{-1}\quad\text{and}\quad
q_i^{2i}=o(M_i)
  \end{gather}
  and $(\mathbf{p}_{q_i, i,M_i},\ell_i)$ is $\mu$-good for the
  approximation scheme $(\nu_i)_{i\geq1}$. Then the sequence
  $\omega=\mathbf{w}_1^{\odot \ell_1}\odot\mathbf{w}_2^{\odot \ell_2}\odot
  \cdots$ is $\mu$-normal.
\end{lem}

\subsection{Normal in base $q$}
Let $A=\{0,1,\ldots,q-1\}$. In this example we take as language the
full-shift $A^*$ and therefore we do not have any restrictions on the
concatenation, \textit{i.e.} $j=0$. Let
\[
\nu(t)=\left\{ \begin{array}{ll}
\frac {1} {q} & \textrm{if $0 \leq t \leq q-1$}\\
0		& \textrm{if $t\geq q$.}
\end{array} \right.
\]
For every $i \in \mathbb{N}$ and $\mathbf{b}=b_1\ldots b_i$, define
$\nu_i(\mathbf{b})=\prod_{t=1}^i \nu(b_t)$. Clearly for $\mathbf{b}\in
A^*$ we have $\mu(\mathbf{b})=q^{-\lvert\mathbf{b}\rvert}$ and
$\nu_i\rightarrow \mu$.

Let $q_i=q$, $M_i=q^{2i} \log i$, $\ell_i=i^{2i}$, and put
$\mathbf{w}_i=\mathbf{p}_{q,i,M_i}$, so $iq^{2i} \log i \leq \lxi \leq
iq^{2i} \log i+iq^i$. A short computation shows that \refeq{good1},
\refeq{good2}, \refeq{good3}, and \refeq{good4} hold with
$\epsilon_i=1/\sqrt{i}$.  Thus, by \refl{mainlemma}, the number whose
digits of its $q$-ary expansion are formed by
$\omega=\mathbf{w}_1^{\odot \ell_1}\odot\mathbf{w}_2^{\odot
  \ell_2}\odot\cdots$ is normal in base $q$.

\subsection{Arbitrary measures}

Let $A=\mathbb{N} \cup \{0\}$ and let $\mu$ be a shift-invariant measure on $A^{\mathbb{N}}$.  
We first need to define a sequence of measures $(\nu_i)$ that converges weakly to $\mu$.  Consider  a word $\mathbf{b}=b_1\ldots b_i$.
If there is an index $n$ such that $b_n>i$, then let
$\nu_i(\mathbf{b})=0$.  Let $S=\{n : b_n=i\}$.  If
$S=\emptyset$, then let $\nu_i(\mathbf{b})=\mu(\mathbf{b})$.  If $S
\neq \emptyset$, then let
$$
\nu_i(\mathbf{b})=\sum_{\mathbf{b}'} \mu(\mathbf{b}'),
$$
where the sum is over all words $\mathbf{b}'=b'_1\ldots b'_k$ such
that for each index $n$ in $S$, $b'_n \geq i$.
Set $$M_i=\ceil{\max\left(i^{2i} \log i,
    \left(\inf\{\mu(\mathbf{b}):\mathbf{b}\in\mathscr{D}_{\nu_i,i}\}\right)^{-1}\right)},$$
$\mathbf{w}_i=\mathbf{p}_{i,i,M_i}$, $j=0$, $\ell_1=1$, and
$$
\ell_i=\ceil{\log i \cdot \max\left(\frac {M_{i+1}+(i+1)^{i+1}}{M_i},\left( \frac {M_{i-1}+(i-1)^{i-1}} {M_i}\right)\cdot i\ell_{i-1}\right)} \hbox{ for }i > 1.
$$
We note that $iM_i \leq \lxi \leq i(M_i+i^i)$, so
\begin{align*}
\frac {\limo} {\li} \cdot \frac { \lxim} { \lxi } \cdot i 
&\leq \frac {\limo} {\li} \cdot \frac {(i-1)\left(M_{i-1}+(i-1)^{i-1}\right)} {iM_i} \cdot i\\
&< \frac {\limo} {\left( \frac {M_{i-1}+(i-1)^{i-1}} {M_i}\right)\cdot i\ell_{i-1} \cdot \log i}\frac {M_{i-1}+(i-1)^{i-1}} {M_i} \cdot i=\frac {1} {\log i} \to 0
%\leq \frac {\limo} {\li} \cdot \frac {(i-1)+b^{i-1}}{1}\\
%\leq \frac {\limo} {q^{i-1} \log i \limo} \cdot 2b^{i-1}=\frac {1} {\log i} \to 0
\end{align*}
and
\begin{align*}
\frac {1} {\li} \cdot \frac { \lxip} { \lxi }
&\leq \frac {1} {\li} \cdot \frac {(i+1)\left(M_{i+1}+(i+1)^{i+1}\right)}{iM_i}\\
&\leq \frac {1} {\frac {M_{i+1}+(i+1)^{i+1}}{M_i}\cdot \log i} \cdot \frac {1+1} {1} \cdot \frac {M_{i+1}+(i+1)^{i+1}}{M_i} 
=\frac {2} {\log i} \to 0.
\end{align*}
%\begin{align*}
%\frac {1} {\li} \cdot \frac { \lxip} { \lxi } \leq \frac {1} {\li} \cdot \frac {(i+1)M_{i+1}+q^{i+1}}{iM_i}
%\leq \frac {1} {\li} \cdot \frac {i+1} {i} \cdot \frac {M_{i+1}} {M_i}+\frac {1} {\ell_i} \cdot q^{i+1}
%\end{align*}
Therefore, conditions \refeq{good1}, \refeq{good2}, \refeq{good3}, and
\refeq{good4} hold with $\epsilon_i=1/\sqrt{i}$.  Thus, by
\refl{mainlemma}, the infinite word $\omega=\mathbf{w}_1^{\odot
  \ell_1}\odot\mathbf{w}_2^{\odot \ell_2}\odot\cdots \cdots$ is
$\mu$-normal.

\subsection{L\"uroth series expansion}\footnote{This example may be
  modified to construct normal numbers with respect to {\it
    Generalized L\"uroth series expansions} (see
  \cite{dajani_kraaikamp2002:ergodic_theory_numbers} for a
  definition of these expansions.)}
Put
\begin{displaymath}
\nu_i(t)=\left\{ \begin{array}{ll}
0 & \textrm{$t=0, 1$}\\
\frac {1} {t(t-1)} & \textrm{$2 \leq t \leq i+1$}\\
\frac {1} {i+1} 		& \textrm{$t=i+2$}\\ % 1-\sum_{t=2}^i \frac {1} {t(t-1)}
0 & t>i+2
\end{array} \right. 
\end{displaymath}
and
\begin{displaymath}
\mu(t)=\left\{ \begin{array}{ll}
0 & \textrm{$i=0, 1$}\\
\frac {1} {t(t-1)} & \textrm{$t \geq 2$}\\
\end{array} \right. 
\end{displaymath}
For $\mathbf{b}=b_1\ldots b_i$, define $\nu_i(\mathbf{b})=\prod_{t=1}^i \nu_i(b_t)$ and $\mu(\mathbf{b})=\prod_{t=1}^i \mu(b_t)$.  
Clearly, $\nu_i \rightarrow \mu$. Next, we let $j=0$, $q_i=i+2$, $M_i=\max(3!^2,i^{2i} \log i)$, $\ell_i=\floor{i^2 \log i}$, and $\mathbf{w}_i=\mathbf{p}_{i+2,i,M_i}$.  Note that for all $i \geq 1$
\begin{equation*}
M_i \geq (i+1)!^2 > \left(\min\{\mu(\mathbf{b}):\mathbf{b}\in\mathscr{D}_{\nu_i,i}\}\right)^{-1}.
\end{equation*}
Since onditions \refeq{good1}, \refeq{good2}, \refeq{good3}, and
\refeq{good4} hold, we deduce by an application of \refl{mainlemma},
that the number whose digits of its L\"uroth series expansion are
formed by $\mathbf{w}_1^{\odot \ell_1}\odot\mathbf{w}_2^{\odot
  \ell_2}\odot\cdots$ is normal with respect to the L\"uroth series
expansion.

This construction has been partially improved (by lowering the number
of repetitions) in a recent paper by Vandehey
\cite{vandehey2013:simpler_normal_number}, who constructed an example
of a normal number for the L\"uroth series expansion.

%, so $(i!)^{2i} \leq \lxi \leq (i!)^{2i}i+(i+2)^i$.  Let $\ell_i=(i!)^2(i+1)^{2i+2} \log i$.  So $x_i$ is $(1/\sqrt{i},\sqrt{i},\nu_i)$-normal. By the main theorem, the numbers whose digits of its L\"uroth series expansion are formed by $\ell_1x_1\ell_2x_2 \cdots$ is normal with respect to the L\"uroth series expansion.

\subsection{Unfair coin}
We note that already Postnikov and Pyatecki{\u\i}
\cite{postnikov_pyateckii1957:bernoulli_normal_sequences} used the
Champernowne word for such a construction. However, since it is an
easy application of \refl{mainlemma} we state this example here for
completeness.

Let $p \in (0,1), p \neq 1/2$.
Here, we consider measures $\nu_i$ where
\begin{displaymath}
\nu_i(t)=\left\{ \begin{array}{ll}
p & \textrm{if $t=0$}\\
1-p & \textrm{if $t=1$}\\
0		& \textrm{if $t>1$}
\end{array} \right. . 
\end{displaymath}
For $\mathbf{b}=b_1\ldots b_i$, let $\nu_i(\mathbf{b})=\prod_{t=1}^i \nu_i(b_t)$ and $\mu=\nu_1$.  Set 
$$M_i=\left(\frac {1} {\min (p,1-p)}\right)^{2i},$$
$j=0$, $\ell_i=i^{2i}$, and put
$\mathbf{w}_i=\mathbf{p}_{2,i,M_i}$. % so $Mi \leq \lxi \leq Mi+2^i$.
Then $\mathbf{w}_i$ is $(1/\sqrt{i},\sqrt{i},\nu_i)$-normal and using
\refl{mainlemma} we get that $\omega=\mathbf{w}_1^{\odot\ell_1}\odot\mathbf{w}_2^{\odot\ell_2}\odot$ is $\mu$-normal.

\subsection{$\beta$-expansions}
Let $\beta>1$. Then every number $x\in[0,1)$ has a greedy
$\beta$-expansions given by the greedy algorithm (\textit{cf.} R{\'e}nyi
\cite{renyi1957:representations_real_numbers}): set $r_0=x$, and for
$n\geq1$, let $d_n=\lf\beta r_{n-1}\rf$ and $r_n=\{\beta r_{n-1}\}$. Then
\[
x=\sum_{n\geq1}d_n\beta^n,
\]
where the $d_n$ are integer digits in the alphabet
$A_\beta=\{0,1,\ldots,\ceil{\beta}-1\}$. We denote by
$\mathrm{d}(x)=d_1d_2d_3\ldots$ the greedy $\beta$-expansion of $x$.

Let $D_\beta$ denote the set of greedy $\beta$-expansions of numbers
in $[0,1)$. A finite (resp. infinite) word is called
$\beta$-admissible if it is a factor of an element (resp. an element)
of $D_\beta$. Not every number is $\beta$-admissible and the
$\beta$-expansion of $1$ plays a central role in the characterization
of all admissible sequences. In particular, let
$\mathrm{d}_\beta(1)=b_1b_2\ldots$ be the greedy $\beta$-expansion of
$1$. Since the expansion might be finite we define the quasi-greedy
expansion $\mathrm{d}^*_\beta(1)$ by
\[
\mathrm{d}^*_\beta(1)=\begin{cases}
(b_1b_2\ldots b_{t-1}(b_t-1))^{\womega}
&\text{if }\mathrm{d}_\beta(1)=b_1b_2\ldots b_t\text{ is finite},\\
\mathrm{d}_\beta(1) &\text{otherwise.}
\end{cases}
\]

Then Parry \cite{parry1960:eta_expansions_real} could show the following
\begin{lem}
Let $\beta>1$ be a real number, and let $s$ be an infinite sequence of
non-negative integers. The sequence $s$ belongs to $D_\beta$ if and only if for
all $k\geq0$
\[
  T^k(s)<d_\beta^*(1),
\]
where $T$ is the shift.
\end{lem}
According to this result we call a number $\beta$ such that
$\mathrm{d}_\beta(1)$ is eventually periodic a Parry number. In the present
example we assume that $\beta$ is such a number.

% Since not all expansions are $\beta$-admissible we have to guarantee that if we
% concatenate the expansions of two words then this will generate an admissible
% sequence. In particular, let $\beta$ be the golden mean, \textit{i.e.}
% $\beta=\frac{1+\sqrt5}2$. Then the expansion of $1$ is equal to
% $d_\beta(1)=11$. In our construction we may take the two words $1001$ and
% $1010$, which are both admissible. However, if we concatenate them, we get the
% word $10011010$, which is not admissible. In order to prevent this, we pad a
% certain amount of zeroes between two words. For the case of the golden mean,
% one zero is sufficient, since then we would get as concatenation
% $1001\,0\,1010$ which is an admissible word.

% In the following we will on the one hand provide an estimate for the number of
% zeroes we have to pad in order to get an admissible sequence. On the other hand
% we have to show that the constructed sequence is a normal number.

For the padding size we denote by $\mathrm{d}_\beta(1)=b_1\ldots
b_t(b_{t+1}\ldots b_{t+p})^\womega$ the $\beta$-expansion of $1$. If
$1$ has a finite expansion then we set $p=0$. We are looking for the
longest possible sequence of zeroes occurring in the expansion of
$1$. As one easily checks, the longest occurs if
$b_1=\cdots=b_{t+{p-1}}=0$ and $b_{t+p}\neq0$. Thus we can set the
padding size $j$ to be
\[
j=t+p.
\]

We wish to minimize the length of a cylinder set defined by a word of length
$\womega$.  Define

\begin{displaymath}
\phi_\beta(\womega)=\left\{ \begin{array}{ll}
1 & \textrm{if $1 \leq \womega \leq t$}\\
r & \textrm{if $t+(r-2)p \leq \womega \leq t+(r-1)p$}\\
\end{array} \right. . 
\end{displaymath}

Then the length of this interval is at least $\beta^{-(t+\phi_\beta(\womega)p)}$.  We use the fact that $\mu_\beta(I) \geq(1-1/\beta) \lambda(I)$ and put
$$
M_i=\max\left(\frac {\beta^{t+\phi_\beta(i)p}} {1-\frac {1} {\beta}}, \ceil{\beta}^{2i} \log i  \right).
$$

%Here, we assume that
%$$
%M(\omega+j)\leq |\mathbf{p}_{b,\omega}| \leq M(\omega+j)+b^{\omega+j},
%$$
%where $j$ is the amount of padding we need to put in.
%Here $j$ is the amount of padding we need to put in.  Is this related to $p$?

Put $\mathbf{w}_i=\mathbf{p}_{\ceil{\beta},i,M_i}$ and $q_i=\ceil{\beta}$. Note that
$\lim_{i \to \infty} \frac {\phi(i)} {i/p}=1$, so for large $i$

$$
(i+j) \ceil{\beta}^{2i} \log i \leq \lxi \leq (i+j)
\left(\ceil{\beta}^{2i}\log i+\ceil{\beta}^i \right)
$$

%$$
%(i+j) \frac {\beta^{t+\phi_\beta(i)p}} {1-\frac {1} {\beta}} \log i
% \leq \lxi \leq (i+j) \left(\frac {\beta^{t+\phi_\beta(i)p}}
%   {1-\frac %{1} {\beta}} \log i+\beta^{\omega+j} \right)
%$$

Thus, for large $i$
$$
\lxi \approx  i \ceil{\beta}^{2i} \log i.
$$
%for $k=\frac {1} {1-1/\beta}\cdot \beta^t$.  
Put $\ell_i=i^{2i}$ and the computation follows the same lines as above.

\subsection{Continued fraction expansion}

For a word $\mathbf{b}=b_1\ldots b_i$, let $\Delta_\mathbf{b}$ be the
set of all real numbers in $(0,1)$ whose first $i$ digits of it's
continued fraction expansion are equal to $\mathbf{b}$.  Put
$$
\mu(\mathbf{b})=\frac {1} {\log 2} \int_{\Delta_\mathbf{b}} \frac {dx} {1+x}.
$$
If there is an index $n$ such that $b_n>i$, then let
$\nu_i(\mathbf{b})=0$.  Let $S=\{n : b_n=i\}$.  For $i<8$, set
$\nu_i(\mathbf{b})=\mu(\mathbf{b})$.  For $i \geq 8$, if
$S=\emptyset$, then let $\nu_i(\mathbf{b})=\mu(\mathbf{b})$.  If $S
\neq \emptyset$, then let
$$
\nu_i(\mathbf{b})=\sum_{\mathbf{b}'} \mu(\mathbf{b}'),
$$
where the sum is over all words $\mathbf{b}'=b'_1\ldots b'_i$ such
that for each index $n$ in $S$, $b'_n \geq i$.

Put
$m_i=\min_{\mathbf{b} \in \mathscr{D}_{\nu_i}, |\mathbf{b}|=i} \nu_i(\mathbf{b}).$
  We wish to find a lower bound for $m_i$.  If $\mathbf{b}=b_1\ldots b_k$, then let $$\frac {p_k} {q_k}=\frac {1} {b_1+\frac {1} {b_2+\ddots+\frac{1} {b_k}}}.$$  It is well known that $\lambda(\Delta_\mathbf{b})=\frac {1} {q_k(q_k+q_{k-1})}$ and $\mu(\mathbf{b}) > \frac {1} {2\log 2} \lambda(\Delta_\mathbf{b})$.

Thus, we may find a lower bound for $m_i$ by minimizing $\frac {1} {q_i(q_i+q_{i-1})}$ for words $\mathbf{b}$ in $\mathscr{D}_{\nu_i}$.  The minimum will occur for $\mathbf{b}=ii\ldots i$.  It is known that $q_n=iq_{n-1}+q_{n-2}$ if we set $q_0=1$ and $q_1=i$. Set
$$
r_1=\frac {i+\sqrt{i^2+4}} {2},\quad r_2=\frac {i-\sqrt{i^2+4}} {2}.
$$
Then
$$
q_n=\frac {r_1^{n+1}-r_2^{n+1}} {\sqrt{i^2+4}}.
$$
Thus,
$$
\frac {1} {q_i(q_i+q_{i-1})}=\frac {i^2+4} {(r_1^{i+1}-r_2^{i+1})((r_1^{i+1}+r_1^i)-(r_2^{i+1}-r_2^i))}>\frac{\log 2} {i^{2i}} \hbox{ for }i \geq 8.
$$
Thus, $m_i > \frac {1} {2\log 2} \left( \frac {\log 2} {i^{2i}} \right)=\frac {1} {2} i^{-2i}$.
Let $M_i = 2i^{2i} \log i$, $j=0$,  $\mathbf{w}_i=\mathbf{p}_{i+1,i,M_i}$. % and %, so that $M \nu_i(\mathbf{b}) >1$. 
%Thus, $2i^{2i+1} \leq \lxi \leq2i^{2i+1}+(i+1)^i$. 
Set $\ell_i=0$ for $i<8$ and $\ell_i=\floor{i^2 \log i}$ for $i \geq 8$.  Then for $i \geq 9$
$$
\frac {\ell_{i-1}} {\ell_i} \frac {|\mathbf{w}_{i-1}|} {\lxi} i< \frac {2(i-1)^{2i-1}+i^{i-1}} {2i^{2i}}=\left(1-\frac {1} {i}\right)^{2i} \frac {1} {i-1}+\frac {1} {2i^{i+1}} \to 0
$$
and
$$
\frac {|\mathbf{w}_{i+1}|} {\ell_i \lxi} \leq \frac {2(i+1)^{2i+3}+(i+2)^{i+1}} {i^2 \log i \cdot 2i^{2i+1}}=\left(1+\frac{1}{i}\right)^{2i} \frac {(i+1)^3} {i^3 \log i} + o(i^{-i}) \to 0.
$$
By \refl{mainlemma} the number whose digits of its continued fraction
expansion are formed by $\mathbf{w}_1^{\odot \ell_1}\odot\mathbf{w}_2^{\odot
  \ell_2}\odot\cdots$ is normal with respect to the continued fraction
expansion.

\section*{Acknowledgment}

Research of the second author is partially supported by the U.S. NSF
grant DMS-0943870.

Parts of this research work were done when the authors were visiting
the Department of Analysis and Computational Number Theory at Graz
University of Technology. Their stay was supported by FWF project
P26114. The authors thank the institution for its hospitality.

% \bibliographystyle{mine}
% \bibliography{literatur}

\providecommand{\bysame}{\leavevmode\hbox to3em{\hrulefill}\thinspace}
\providecommand{\MR}{\relax\ifhmode\unskip\space\fi MR }
% \MRhref is called by the amsart/book/proc definition of \MR.
\providecommand{\MRhref}[2]{%
  \href{http://www.ams.org/mathscinet-getitem?mr=#1}{#2}
}
\providecommand{\href}[2]{#2}

\end{document}